\documentclass [12pt,a4paper]{amsart}
\usepackage{amsmath}
\usepackage{stmaryrd}
\usepackage{textcomp}
\usepackage{cases}
\usepackage{amscd}
\usepackage{epsfig}
\usepackage{graphicx}

\usepackage{color}

\usepackage{verbatim}

\usepackage{amssymb}
\usepackage{amsfonts}
\usepackage{amsbsy}

\usepackage{amsthm}
\usepackage{amstext}
\usepackage{amsopn}

\newcommand{\bX}{{\bf X}}
\newcommand{\bH}{{\bf H}}
\newcommand{\bY}{{\bf Y}}

\newcommand{\bT}{{\bf T}}

\newcommand{\HC}{{\textbf{H}_{\mathbb{C}}^2}}
\newcommand{\PU}{\textrm{PU}}

\newcommand{\X}{\mathbb{X}}
\newcommand{\calL}{\mathcal{L}}

\newcommand{\calB}{\mathcal{B}}
\newcommand{\calC}{\mathcal{C}}
\newcommand{\calH}{\mathcal{H}}

\newcommand{\calT}{\mathcal{T}}

\newcommand{\A}{\mathbb{A}}
\newcommand{\C}{\mathbb{C}}
\newcommand{\R}{\mathbb{R}}

\newcommand{\J}{\mathbb{J}}
\newcommand{\frakC}{\mathfrak{C}}
\newcommand{\frakF}{\mathfrak{F}}

\newcommand{\frakH}{\mathfrak{H}}
\newcommand{\frakV}{\mathfrak{V}}
\newcommand{\frakX}{\mathfrak{X}}
\newcommand{\fp}{\mathfrak{p}}
\newcommand{\bp}{{\bf p}}
\newcommand{\bz}{{\bf z}}
\newcommand{\bZ}{{\bf Z}}
\newcommand{\bW}{{\bf W}}

\newcommand{\bS}{{\bf S}}
\newtheorem{thm}{Theorem}[section]
\newtheorem{defi}[thm]{Definition}
\newtheorem{cor}[thm]{Corollary}

\newtheorem{prop}[thm]{Proposition}

\usepackage{fancyhdr}
\begin{document}
\title[$\PU(2, 1)$ configuration space of $4$ points]{A K\"ahler structure for the $\PU(2, 1)$ configuration space of four points in $S^3$}
\author{Ioannis D. Platis \and Li-jie Sun}
\thanks{Corresponding author: lijiesun@sdu.edu.cn.}

\address{Department of Mathematics and Applied Mathematics,
University of Crete,
 Heraklion Crete 70013,
 Greece.}
\email{jplatis@math.uoc.gr}
\address {School of Mathematics and Statistics, 
Shandong University, Weihai, 
Shandong 264209,
P. R. China.}
\email{lijiesun@sdu.edu.cn}

\begin{abstract}
We show that an open subset $\frakF_4''$ of the ${\rm PU}(2,1)$ configuration space of four points in $S^3$ is in bijection with an open subset of %with a K\"ahler structure which is inherited from the one of 
$\frakH^{\star}\times\R_{>0}$, where $\frakH^\star$ is the affine-rotational group. Since the latter is a Sasakian manifold, the cone $\frakH^\star\times\R_{>0}$ is K\"ahler and thus $\frakF_4''$ inherits this K\"ahler structure.
\end{abstract}
\maketitle

%\tableofcontents
\section{Introduction}

Moduli spaces of $n$-tuples of points in the boundary of symmetric spaces of rank 1 and of non compact type (that is, hyperbolic spaces) are of great interest. Many times these spaces have remarkable geometric properties and they are interesting on their own. In other cases they appear in the study of deformations spaces of topological surfaces, where they serve as parameter spaces, see for instance \cite{PP0}. There, Fenchel-Nielsen type coordinates are established for the space of discrete, faithful and totally loxodromic representations of the fundamental group of a closed surface of negative Euler characteristic into ${\rm SU}(2,1)$ (that is, the triple cover of the holomorphic isometry group ${\rm PU}(2,1)$ of the complex hyperbolic plane $\bH^2_\C$). Incorporated into these coordinates are specific complex parameters, the {\it Kor\'anyi-Reimann complex cross-ratios}, see Section \ref{sec-X} for details; these parameters are directly related to the configuration space of four points in the boundary of $\bH^2_\C$. The boundary $\partial\bH^2_\C$ is in turn identified to $S^3$ or, to the one point compactification of the Heisenberg group. Recall that the Heisenberg group $\frakH$ is the Lie group with underlying manifold $\C\times\R$ and multiplication given by
$$
(z,t)\cdot(w,s)=\left(z+w,t+s+2\Im(z\overline{w})\right),
$$
for every $(z,t),(w,s)\in\C\times\R$; for details, see Section \ref{sec:heis}.

By $\frakC_4$ we shall denote the space of ordered quadruples $\fp=(p_1,p_2,p_3,p_4)$ of pairwise distinct points in $S^3$.  The group ${\rm PU}(2,1)$ acts diagonally on $\frakC_4$; we denote the quotient by $\frakF_4$: this is the {\it configuration space of four points in $S^3$} and it has been studied rather extensively, see for instance, \cite{CG,FP,PP1,Pla-CR}. We consider the following subsets of $\frakC_4$:
\begin{enumerate}
\item[{(i)}] The subset $\frakC^\R_4$ comprising quadruples $\fp$ such that $p_i$ do not all lie in the same $\C$-circle.
\item[{(ii)}] The subset $\frakC_4'\subset\frakC^\R_4$ comprising quadruples $\fp$ such that $p_2,p_3,p_4$ do not lie in the same $\C$-circle.
\item[{(iii)}] The subset $\frakC^\C_4\subset\frakC^\R_4$ comprising quadruples $\fp$ such that $p_1$ and $p_4$ do not lie in the same orbit of the stabiliser of $p_2$ and $p_3$.
\end{enumerate}
In \cite{FP} it has been proved that
\begin{enumerate}
\item[{(i)}] The subset $\frakF^\R_4$ is a 4-dimensional manifold.
\item[{(ii)}] The subset $\frakF_4'\subset\frakF^\R_4$ admits a {\rm CR} structure of codimension 2.
\item[{(iii)}] The subset $\frakF^\C_4\subset\frakF^\R_4$ is a 2-dimensional complex manifold, biholomorphic to $\C P^1\times(\C\setminus\R)$.
\end{enumerate}
All the above results are obtained via the identification of $\frakF^\R_4$ to the {\it cross-ratio variety}, see Section \ref{sec-X}. Platis also proved in \cite{Pla-CR} that  $\frakF^\C_4$ admits another complex structure. This complex structure and the complex structure in (iii) agree on the CR structure and they are opposite on its complement in the holomorphic tangent bundle of $\frakF_4^{\C}$.

In this paper, we consider a new subset $\frakC_4''$ of $\frakC^\C_4$ satisfying:
\begin{itemize}
\item
 $p_1,p_2,p_3$ as well as $p_2,p_3,p_4$ do not all lie in the same $\C$-circle.
\item $p_1$ and $p_4$ are not in the same orbit of the stabiliser of $p_2$ and $p_3$.
\end{itemize}
%A natural question that arises here, is whether the complex structure of $\frakF^\C_4$ can be also defined on the subset $\frakF_4'$ which by itself is a  CR manifold. To this direction, 
We wish to consider the geometric structures of the configuration space $\frakF_4''$. To this direction, we set $\fp=(p_1,p_2,p_3, p_4)\in\frakC_4''$. We may normalise so that 
$$
p_1=(1,\tan a),\quad p_2=\infty,\quad p_3=(0,0),\quad p_4=(z,t),
$$
where $a\in(-\pi/2,\pi/2)$, $z\neq 0$ and $t/|z|^2\neq\tan a$, see the details in Section \ref{sec-mainth}. 
Then we may define a map from 
$\frakF_4''$ to $\C_*\times\R\times\R_{>0}$, which maps $[\fp]$ to $(z,t,e^{\tan a})$.
The set $\C_{\ast}\times\R$ can be given the structure of a Lie group $\frakH^{\star}$, which is actually the {\it affine-rotational group}, see Section \ref{sec:hypheis}. The group $\frakH^{\star}$ can be endowed with a contact structure $\omega^\star$, see Proposition \ref{prop:omegasta}. With this contact structure, $\frakH^{\star}$ is a contact open submanifold of the Heisenberg group $\frakH$ with its natural contact form $\omega=dt+2xdy-2ydx.$ The strictly Levi pseudoconvex CR structure associated to  $\omega^{\star}$ is compatible to the strictly Levi pseudoconvex CR structure of the truncated boundary of complex hyperbolic plane, that is, $\bH^{\star}=\partial\bH^{2}_{\C}\setminus\partial\bH^{1}_{\C},$ see Section \ref{sec:con}. Using standard methods of contact Riemannian geometry, we show for clarity in Section \ref{sec:sashhg} that $\frakH^{\star}$ is Sasakian (Theorem \ref{thm:HstarSas}). This automatically implies that its Riemannian cone $\calC(\frakH^{\star})$ is a K\"ahler manifold, see Section \ref{sec:kcon}. We consider the subset $\calC'(\frakH^{\star})=\calC(\frakH^{\star})\setminus\{(z,t,r)\;|\;\log r=t/|z|^2\}$, which is an open subset of $\calC(\frakH^{\star})$.  In this manner we obtain our main theorem:

\medskip

\noindent{\bf Theorem \ref{thm-main}}
{\it  There is a bijection $\mathcal B_0: \frakF_4''\to\calC'(\frakH^{\star})$. Therefore $\frakF_4''$ inherits the K\"ahler structure of 
$\mathcal C'(\frakH^{\star})$. }

\medskip 

This paper is organised as follows: After the preliminaries in Section \ref{sec:prel}, we study the affine-rotational group $\frakH^\star$ and show its Sasakian structure in Section \ref{sec:hhg}; the K\"ahler structure of $\calC(\frakH^\star)$ is studied in Section \ref{sec:kcon}. In Section \ref{sec-conf} we prove our main result; the rest of the paper is devoted to describing the CR  structure and the complex structure of $\frakF_4''$ and showing that they are CR compatible with the respective structures that have been established in \cite{FP}.

\section{Preliminaries}\label{sec:prel}
In this section we review {\rm CR}, contact and Sasakian structures (Section \ref{sec:CRetc}), the basics about complex hyperbolic plane and its boundary (Section \ref{sec:chp}) and finally the Heisenberg group and its Sasakian geometry (Section \ref{sec:heis}).

\subsection{{\rm CR}, contact and  Sasakian structures}\label{sec:CRetc}
The material of this section is standard. We refer for instance to \cite{Be}, \cite{Bo-Ga}, for further details.

Let $M$ be a $(2p+s)$-dimensional real manifold. A codimension $s$ {\rm CR} structure in $M$ is a pair $(\calH, J)$ where $\calH$ is a $2p$-dimensional smooth subbundle of ${\rm T}(M)$ and $J$ is an almost complex endomorphism of $\calH$ which is formally integrable: If $X$ and $Y$ are sections of $\calH$ then the same holds for 
  $\left[X, Y\right]-\left[JX, JY\right], \left[JX, Y\right]+\left[X, JY\right]$ and moreover, 
  $J(\left[X, Y\right]-\left[JX, JY\right])=\left[JX, Y\right]+\left[X, JY\right]$.

If $s=1$, a contact structure on $M$ is a codimension 1 subbundle $\calH$ of ${\rm T}(M)$ which is completely non-integrable; alternatively, $\calH$ may be defined as the kernel of a 1-form $\eta$, called the contact form of $M$, such that $\eta\wedge (d\eta)^{p}\neq 0$. The dependence of $\calH$ on $\eta$ is up to multiplication of $\eta$ by a nowhere vanishing smooth function. By choosing an almost complex structure $J$ defined in $\calH$ we obtain a {\rm CR} structure $(\calH, J)$ of codimension 1 in $M$. The subbundle $\calH$ is also called the horizontal subbundle of ${\rm T}(M)$. The closed form $d\eta$ endows $\calH$ with a symplectic structure and we may demand from $J$ to be such that $d\eta(X,JX)>0$ for each $X\in\calH$; we then say that $\calH$ is strictly pseudoconvex. The Reeb vector field $\xi$ is the vector field which satisfies
$
\eta(\xi)=1$ and $\xi\in\ker(d\eta).
$
Note that the Reeb vector field is uniquely determined by the contact form.
%By the contact version of Darboux's Theorem, $\xi$ is unique up to change of coordinates.

Neat examples of contact structures on 3-dimensional  manifolds are the strictly pseudoconvex {\rm CR} structures on boundaries of domains in $\C^2$. Let $D\subset \C^2$ be a domain with defining function $\rho:D\to\R_{>0}$, $\rho=\rho(z_1,z_2)$. On the boundary $M=\partial D$ we consider the form $d\rho$; if $J$ is the complex structure of $\C^2$ we then let
$$
\eta=-\Im(\partial\rho)=-\frac{1}{2}Jd\rho.
$$
We thus obtain the {\rm CR} structure $(\calH=\ker(\eta),J)$.  This is a contact structure if and only if the Levi form 
$
L=d\eta=i\partial\overline{\partial}\rho
$
is positively oriented.

Let $(M, \eta)$ be a $(2p+1)$-dimensional pseudo-hermitian manifold equipped with a CR structure $(\calH=\ker(\eta),J).$ 
%Let now $(M,\calH,J)$ be a contact manifold with $\dim(M)=2p+1$. 
The almost complex structure $J$ on $\calH$ is then extended to an endomorphism $\phi$ of the whole tangent bundle ${\rm T}(M)$ by setting $\phi(\xi)=0$. Subsequently, a canonical Riemannian metric $g$ is defined in $M$ from the relations
\begin{equation}\label{eq:contactmetric}
\eta(X)=g(X,\xi),\quad \frac{1}{2}d\eta(X,Y)=g(\phi X, Y),\quad \phi^2(X)=-X+\eta(X)\xi,
\end{equation}
for all vector fields $X, Y$ on $M$.
We then call $(M;\eta,\xi,\phi,g)$ the contact Riemannian structure  on $M$ associated to the pseudo-hermitian structure $(M, \eta).$ If $f:M\to M$ is an automorphism which preserves the contact Riemannian structure, then one may use equations (\ref{eq:contactmetric}) to verify straightforwardly that this happens if and only if $f$ is {\rm CR}, that is $f_*J=Jf_*$ and also $f^{*}\eta=\eta$.

A contact Riemannian manifold for which the Reeb vector field $\xi$ is Killing (equivalently, $\xi$ is an infinitesimal {\rm CR} transformation) is called a  K-contact Riemannian manifold. 

Consider now the Riemannian cone $\calC(M)=(M\times\R_{>0},\;g_r=dr^2+r^2g)$. We may define an almost complex structure $\J$ in $\calC(M)$ by setting
$$
\J X=JX,\quad X\in\calH(M),\quad \J(r\partial_r)=\xi.
$$
The fundamental 2-form for $\calC(M)$ is then the exact form
$$
\Omega_r=d\left(\frac{r^2}{2}\eta\right)=r\;dr\wedge\eta+\frac{r^2}{2}\;d\eta,
$$
and therefore it is closed. We have then that $(M;\eta,\xi,\phi,g)$ is Sasakian if and only if the {\it Riemannian cone} $(\calC(M);\J,g_r,\Omega_r)$ is K\"ahler. The following proposition is often useful.
\begin{prop}\label{prop:S-c}
Let $(\eta,\xi,\phi,g)$ be a K-contact Riemannian structure on $M$. Then $M$ is a Sasakian manifold if and only if the contact Riemannian structure satisfies
$$
R(X,\xi)Y=g(X,Y)\xi-g(\xi,Y)X,
$$
for any vector fields $X,Y$ on $M$. Here
%$\nabla$ is the Riemannian connection and 
$R$ is the Riemannian curvature tensor of $g$.
\end{prop}

\medskip

We wish to further comment on the sub-Riemannian geometry of a contact and a contact Riemannian manifold, respectively. If $(M, \eta)$ is a $(2p+1)$-dimensional pseudo-hermitian manifold equipped with a CR structure $(\calH=\ker(\eta),J)$, we may first define a Riemannian metric $g_{cc}$ in $\calH$ (the sub-Riemannian metric);  the distance $d_{cc}(p,q)$ between two points $p,q$ of $M$ is given by the infimum of the $g_{cc}$-length of horizontal curves joining $p$ and $q$. By a horizontal curve $\gamma$ we mean a piece-wise smooth curve in $M$ such that $\dot\gamma\in\calH$. The metric $d_{cc}$ is the Carnot-Carath\'eodory metric and there are two interesting facts about it: Firstly, the metric topology coincides with the manifold topology and secondly, if $g_{cc}'$ is another sub-Riemannain metric, then $d_{cc}$ and $d_{cc}'$ are bi-Lipschitz equivalent on compact subsets of $M$. In the case where we construct a contact Riemannian structure $(M;\eta,\xi,\phi,g)$ out of a pseudo-hermitian structure $(M, \eta)$ as above, the sub-Riemannian metric $g_{cc}$ may be taken as the restriction of $g$ into $\calH\times\calH$, i.e., $g=g_{cc}+\eta\otimes\eta$. If $d_g$ is the Riemannian distance corresponding to the Riemannian metric $g$ and $d_{cc}$ is the Carnot-Carath\'eodory distance corresponding to $g_{cc}$, then we always have $d_g\le d_{cc}$. It also follows that the group ${\rm Aut}(M)$ of automorphisms of the contact Riemannian structure $g$ is just the group ${\rm Isom}_{cc}(M)$ of isometries of $d_{cc}$. If the contact Riemannian structure is Sasakian, then the group ${\rm Aut}(\calC(M))$ of automorphisms of $\calC(M)$ is just ${\rm Isom}_{cc}(M)$.

\subsection{Complex hyperbolic plane}\label{sec:chp}
Let $\C^{2, 1}$ be a 3-dimensional $\C$-vector space equipped with a Hermitian form of signature $(2, 1)$.  For the purpose of our paper we shall work with the Hermitian form given by the matrix 
\begin{equation*}
H=\left(\begin{matrix}
  0 &\quad 0 &\quad 1\\
 0&\quad1&\quad0\\
  1&\quad0&\quad0
\end{matrix}\right).
\end{equation*}
Thus $\langle {\bf z, w}\rangle=\overline{{\bf w}}^{t}H{\bf z}=z_1\overline{w_3}+z_2\overline{w_2}+z_3\overline{w_1},$
where ${\bf z}=[z_1, z_2, z_3]^t$ and ${\bf w}=[w_1, w_2, w_3]^t$. 

Complex hyperbolic plane $\HC$ is the projectivisation in $\C^2$ of negative vectors in $\C^{2,1}$, that is, vectors $\bz$ such that $\langle\bz,\bz\rangle<0$. The resulting domain, described in coordinates $(z_1,z_2)$ of $\C^2$ by
\begin{equation}\label{eq:Sieg}
\rho(z_1,z_2)=2\Re(z_1)+|z_2|^2<0,
\end{equation}
is the \emph{Siegel domain} model for $\HC$. This is a K\"ahler manifold, and its K\"ahler metric is the Bergman metric.
The boundary $\partial\HC$ of complex hyperbolic plane is the projectivisation of null vectors of $\C^{2,1}$, that is, vectors $\bz$ such that $\langle\bz,\bz\rangle=0$. It is identified to the one point compactification of the boundary of the Siegel domain, that is $S^{3}$. On the other hand, the boundary of the Siegel domain is naturally identified to the 3-dimensional \emph{Heisenberg group} $\mathfrak{H}$. This is the set $\mathbb{C}\times\mathbb{R}$ with the group law
$$(\zeta_{1},t_{1})\cdot(\zeta_{2},t_{2})=
(\zeta_{1}+\zeta_{2},t_{1}+t_{2}+2\Im(\zeta_{1}
\overline{\zeta_{2}})).$$
Coordinates for $\mathfrak{H}$ are thus $(z,t)$ and therefore $\partial\HC$ may be viewed as the set comprising {\it standard lifts} of points $(z,t)\in\mathfrak{H}$, that is,
\begin{displaymath}
\left[ \begin{array}{ccc}
-|z|^2+it\\
\sqrt{2}\;z\\
1
\end{array} \right],
\end{displaymath}
and the point at infinity $\infty$ which shall be
$[1,0,0]^{t}$.

Any totally geodesic subspace in $\HC$ is one of the the following types:
\begin{enumerate}
\item[{(i)}] Complex geodesic, which is an isometrically embedded copy of $\textbf{H}^{1}_{\mathbb{C}}$. It has the Poincar\'e model of hyperbolic geometry with constant curvature $-1$;
\item[{(ii)}] Totally real Lagrangian plane, which is an isometrically embedded copy of $\textbf{H}^{2}_{\mathbb{R}}$. It has the Beltrami-Klein projective model with constant curvature $-1/4$.
\end{enumerate}

The intersection of a complex geodesic $L$ with $\partial \HC$ is called a \emph{$\mathbb{C}$-circle}. Correspondingly, the intersection of a totally real Lagrangian plane with $\partial\HC$ is called an \emph{$\mathbb{R}$-circle}.
For more details we refer for instance to \cite{G}.

\subsection{Heisenberg group}\label{sec:heis}
The Heisenberg group $\frakH$ is a 2-step nilpotent Lie group. Consider the left-invariant vector fields
\begin{eqnarray*}
X=\frac{\partial}{\partial x}+2y\frac{\partial}{\partial t},\quad Y=\frac{\partial}{\partial y}-2x\frac{\partial}{\partial t},
\quad T=\frac{\partial}{\partial t}.
\end{eqnarray*}
The vector fields $X,Y,T$ form a basis for the Lie algebra $\mathfrak{h}$ of $\frakH$; this has a grading $\mathfrak{h} = \mathfrak{v}_1\oplus \mathfrak{v}_2$ with
\begin{displaymath}
\mathfrak{v}_1 = \mathrm{span}_{\R}\{X, Y\}\quad \text{and}\quad \mathfrak{v}_2=\mathrm{span}_{\R}\{T\}.
\end{displaymath} 
It is well known that Heisenberg group $\frakH$ admits a strictly pseudoconvex CR structure with contact form $\omega=d t+ 2 \Im(\bar{z}dz)$, and the Reeb vector field for $\omega$ is $T$. Following the strategy described in Section \ref{sec:CRetc}, one may define a contact Riemannian structure on $\frak{H}$. 
The endomorphism $\Phi$ of ${\rm T}(\frakH)$ is given by
$$
\Phi(X)=JX=Y,\quad\Phi(Y)=JY=-X,\quad\Phi(T)=0,
$$
and the Riemannian tensor for $g$ may be written in Cartesian coordinates as
\begin{equation*}\label{eq:gH}
g=g_{cc}+\omega^2=dx^2+dy^2+(dt+2x\;dy-2y\;dx)^2.
\end{equation*}
It can be then shown that $(\frakH;\omega,T,\Phi,g)$ is Sasakian, see for instance \cite{Bo}.

\section{The affine-rotational group and its Sasakian structure}\label{sec:hhg}
In Section \ref{sec:hypheis} we identify the affine-rotational group ${\rm Aff}(\R)\times{\rm U}(1)$ with a subset $\frakH^\star$ of the Heisenberg group $\frakH$. A strictly pseudoconvex CR structure on $\frakH^\star$ is inherited by its identification of $\frakH^\star$ with the boundary of the truncated complex hyperbolic plane $\bH^{\star}$, see Section \ref{sec:con}. We identify the Sasakian structure of $\frakH^\star$ in Section \ref{sec:sashhg}; this structure is actually equivalent to the Sasakian structure of the unit tangent bundle of the hyperbolic plane.  Finally, in Section \ref{sec:kcon} we describe in detail the K\"ahler structure of the Riemannian cone $\calC(\frakH^\star)=\frakH^\star\times\R_{>0}$. 

\subsection{Affine-rotational group}\label{sec:hypheis}
\begin{defi}
We define the group $\frakH^\star$ to be $\C_*\times\R$ with multiplication rule
$$
(z,t)\star(w,s)=(zw,\;t+s|z|^2).
$$
\end{defi}
One verifies straightforwardly that $\frakH^\star$ is a non-Abelian group; the unit element of $\frakH^\star$ is (1,0) and the inverse of an arbitrary $(z,t)\in\frakH^\star$ is $(1/z,\;-t/|z|^2)$. The group $\frakH^\star$ is a Lie group with underlying manifold $\C_*\times\R$; indeed, the map
\begin{eqnarray*}
&&
\frakH^\star\times\frakH^\star\ni\left((z,t),\;(w,s)\right)\mapsto (z,t)^{-1}\star (w,s)=\left(\frac{w}{z},\;\frac{-t+s}{|z|^2}\right)\in\frakH^\star,
\end{eqnarray*}
is clearly smooth.

Note that we can identify $\frakH^\star$ with the group product ${\rm Aff}(\R)\times{\rm U}(1)$, where 
$$
{\rm Aff}(\R)=\left\{\left[\begin{matrix}
\alpha&\beta\\
0&1\end{matrix}\right],\quad \alpha>0,\;\beta\in\R\right\}
$$
is the orientation preserving affine group of the real line. The identification is via the map $\psi:\frakH^\star\to {\rm Aff}(\R)\times{\rm U}(1)$ by
$$
\psi (z,t)=\left(\left[\begin{matrix}
|z|^2\quad&t\\
0\quad&1\end{matrix}\right],\;e^{i\arg z}\right),\quad(z,t)\in\frakH^\star.
$$
One can check straightforwardly that $\psi$ is a Lie group diffeomorphism. 

We next consider the set
$$
\bH^{\star}=\partial\bH^2_\C\setminus\partial\bH^1_\C.
$$
The set $\bH^{\star}$ comprises points $(z_1,z_2)$ of $\C^2$ such that
$$
\rho^{\star}(z_1,z_2)=\frac{2\Re(z_1)}{|z_2|^2}+1=\frac{\rho(z_1,z_2)}{|z_2|^2}=0,
$$
where $\rho$ is the defining function of $\bH^2_\C$ as in (\ref{eq:Sieg}). The bijection $\Psi:\frakH^{\star}\to\bH^{\star}$ is given by
$$
\Psi(z,t)=(-|z|^2+it,\sqrt{2}z),\quad (z,t)\in\frakH^\star
$$
is a diffeomorphism. In conclusion, we have the following
\begin{prop}
The group $\frakH^\star$ is identified to the affine-rotational group ${\rm Aff}(\R)\times{\rm U}(1)$ and to the set $\bH^\star$.
\end{prop}
From now on we shall refer to $\frakH^\star$ as the affine-rotational group. 
\subsection{Contact structure}\label{sec:con}
We may first endow $\frakH^\star$ with a strictly pseudoconvex CR structure which comes from the one of $\bH^\star$ that is described in the following:
\begin{prop}
There is a strictly pseudoconvex CR structure on $\bH^{\star}$. 
\end{prop}
\begin{proof}
It follows from $\rho^{\star}(z_1,z_2)=\frac{\rho(z_1,z_2)}{|z_2|^2}$ that
$$\partial \rho^{\star}=\frac{dz_1+\overline{z_2}dz_2}{|z_2|^2},\qquad 
\bar{\partial} \rho^{\star}=\frac{d\overline{z_1}+z_2d\overline{z_2}}{|z_2|^2}.$$
A CR structure is defined by the (1, 0) vector field
$Z=-|z_2|^2\frac{\partial}{\partial z_1}+z_2\frac{\partial}{\partial z_2}$, and direct calculation yields to
$$\partial\bar{\partial} \rho^{\star}=-\frac{\overline{z_2}}{|z_2|^4}dz_2\wedge d\overline{z_1},$$ which indicates that 
$$
\partial\bar{\partial}{ \rho}^{\star}(Z, \bar{Z})=-\frac{\overline{z_2}}{|z_2|^4}dz_2\wedge d\overline{z_1}\left(-|z_2|^2\frac{\partial}{\partial z_1}+z_2\frac{\partial}{\partial z_2}, -|z_2|^2\frac{\partial}{\partial \overline{z_1}}+\overline{z_2}\frac{\partial}{\partial \overline{z_2}}\right)\\
=1>0.
$$
Therefore the Levi form is positively oriented on the CR structure.
The contact form is
$$
\eta^{\star}=\Im(\partial \rho^{\star})=\frac{dy_1+\Im(\overline{z_2}dz_2)}{|z_2|^2}.
$$
The proof is complete.
\end{proof}
Using the map $\Psi$ and the previous proposition we immediately have:
\begin{prop}\label{prop:omegasta}
There exists a strictly pseudoconvex CR structure on $\frakH^\star$. The CR structure is given by the vector field
$$
\bZ=\Psi^{-1}_{*}(Z)=z\frac{\partial}{\partial z}+i|z|^2\frac{\partial}{\partial t}
$$
and the contact form $\omega^\star$ by
$$
\omega^\star=\Psi^{*}\eta^{\star}=\frac{\omega}{2|z|^2},
$$
where $\omega$ is the contact form of the Heisenberg group $\frak H$.
\end{prop}
Note that
$$
\bZ=zZ,\quad \overline{\bZ}=\overline{z}\overline{Z},
$$
where $Z=(1/2)(X-iY)$, $\overline{Z}=(1/2)(X+iY)$, and $X,Y$ are the generators of the CR structure of $\frakH$, see Section \ref{sec:heis}. By setting 
$$
\bZ=\frac{1}{2}(\bX-i\bY),\quad\overline{\bZ}=\frac{1}{2}(\bX+i\bY),
$$ 
we have that the CR structure $\calH^\star=\ker(\omega^\star)$ is generated by the vector fields 
\begin{equation}\label{eq:basis}
\bX=x\frac{\partial}{\partial x}+y\frac{\partial}{\partial y},
 \quad
 \bY=x\frac{\partial}{\partial y}-y\frac{\partial}{\partial x}-2|z|^2\frac{\partial}{\partial t},
\end{equation}
and also, if $J$ is the alsmost complex structure in $\calH^*$ then
$$
J\bX=\bY,\quad J\bY=-\bX.
$$ 
By writing 
$$
\omega^*=\frac{dt}{2|z|^2}+d\arg z,
$$
we obtain
$$
d\omega^*=2\left(\frac{d(|z|^2)}{2|z|^2}\right)\wedge \left(-\frac{dt}{2|z|^2}\right).
$$
The vector field
\begin{equation}
 \bT=x\frac{\partial}{\partial y}-y\frac{\partial}{\partial x}
\end{equation}
is the Reeb vector field for $\omega^\star$: $\bT\in\ker(d\omega^\star)$ and $\omega^\star(\bT)=1$. The left-invariant vector fields $\bX$, $\bY$, $\bT$ form a basis for the tangent space and the only non-trivial Lie bracket relation is
$$
[\bX,\bY]=2(\bY-\bT).
$$
The dual basis to $\{\bX,\bY,\bT\}$ is $\{\phi^*,\psi^*,\omega^\star\}$,
where
\begin{equation}\label{eq:basis-cotr}
\phi^\star=\frac{d(|z|^2)}{2|z|^2},\quad \psi^\star=-\frac{dt}{2|z|^2}.
\end{equation}
In this basis,
$$
d\phi^\star=0,\quad d\psi^\star=-2\;\phi^*\wedge\psi^*,\quad d\omega^\star=2\;\phi^\star\wedge\psi^\star.%,\quad %dm=\omega^\star\wedge\phi\wedge\psi.
$$
For the symplectic form $d\omega^\star$ we have
$$
d\omega^\star(\bX,\bY)=2,\quad d\omega^\star(\bX,\bT)=d\omega^\star(\bY,\bT)=0.
$$
Finally, if 
 $dm$ is the Haar measure for $\frakH^\star$, then 
$$
\omega^\star\wedge d\omega^\star=\frac{dt\wedge dx\wedge dy}{|z|^4}=-dm.
$$
\subsection{Sasakian structure}\label{sec:sashhg}
A contact Riemannian structure on $\frakH^\star$ is established as follows: Let $\{\bX,\bY,\bT\}$ be the basis for the tangent space  and let also 
$\{\phi^\star,\psi^\star,\omega^\star\}$ be the dual basis as in (\ref{eq:basis-cotr}). By using equations (\ref{eq:contactmetric}) we verify straightforwardly that the Riemannian metric which we shall denote by $g^{\star}$ obtained out of the endomorphism of the tangent space (we shall denote by $\Phi^*$) can be given by declaring the basis $\{\bX,\bY,\bT\}$ orthonormal. Hence, 
$(\frakH^\star;\omega^\star,\bT,\Phi^\star,g^\star)$ is  a contact Riemannian manifold.
The Riemannian tensor is also written as
\begin{equation}\label{eq-Riem}
g^\star=ds^2=(\phi^\star)^2+(\psi^\star)^2+(\omega^\star)^2=\frac{(d(|z|^2))^2+dt^2}{4|z|^4}+\frac{(dt+2xdy-2ydx)^2}{4|z|^4}.
\end{equation}
The restriction of $g^\star$ into $\calH^\star=\{\bX,\bY\}$, that is,
$$
g^\star_{cc}=(\phi^\star)^2+(\psi^\star)^2=\frac{(d(|z|^2))^2+dt^2}{4|z|^4},
$$
defines the K\"ahler structure on the horizontal tangent bundle $\calH^\star$. {Note that  $g^\star_{cc}$ also can be obtained as the symmetric 2-form from the hyperbolic metric via the map $(z,t)\mapsto -|z|^2+it$  (compare with \cite{G}, Section 4.3.6}). We shall prove:
%which is the pullback from the left half hyperbolic plane of the hyperbolic metric to $\frakH^{\star}$, via the map $(z,t)\mapsto -|z|^2+it$ which is the composition of (compare with \cite{G}).} We shall prove:
\begin{thm}\label{thm:HstarSas}
$(\frakH^\star;\omega^\star,\bT,\Phi^\star,g^\star)$ is Sasakian.
\end{thm}
\begin{proof}
If $\nabla$ is the Riemannian connection of $g^\star$, then using Koszul's formula \begin{equation*}\label{eq:Koszul}
-2g(Z,\nabla_YX)=g([X,Z],Y)+g([Y,Z],X)+g([X,Y],Z),
\end{equation*} we have:  
\begin{eqnarray*}
&&
\nabla_\bX\bX=0,\quad \nabla_\bY\bX=-2\bY+\bT,\quad \nabla_\bT\bX=\bY,\\
&&
\nabla_\bX\bY=-\bT,\quad \nabla_\bY\bY=2\bX,\quad \nabla_\bT\bY=-\bX,\\
&&
\nabla_\bX\bT=\bY,\quad \nabla_\bY\bT=-\bX,\quad \nabla_\bT\bT=0.
\end{eqnarray*}
Denote by $R$ the curvature tensor,
\begin{equation*}\label{eq:Curvten}
R(X,Y)Z=\nabla_Y\nabla_XZ-\nabla_X\nabla_YZ+\nabla_{[X,Y]}Z.
\end{equation*} 
Then 
\begin{eqnarray}\label{eq:CT}
&&\notag
R(\bX,\bY)\bX=-7\bY,\quad
R(\bX, \bT)\bX=\bT,\quad R(\bY, \bT)\bX=0,\\
&&
R(\bX, \bY)\bY=\bX,\quad
R(\bX, \bT)\bY=0,\quad R(\bY, \bT)\bY=\bT,\\
&&\notag
R(\bX, \bY)\bT=4\bX,\quad
R(\bX, \bT)\bT=-\bX,\quad R(\bY, \bT)\bT=-\bY.
\end{eqnarray}
To prove that $\frakH^{\star}$ is $K$-contact, that is, $\bT$ is Killing, we show that
$$
(\calL_\bT g^\star)(U,V)=g^\star(\nabla_V\bT, U)+g^\star(V,\nabla_U\bT)=0,
$$
for any vector fields $U, V$ on $\frakH^{\star}$.
Taking arbitrary vector fields $U=a\bX+b\bY+c\bT, V=a'\bX+b'\bY+c'\bT$, we get that
$$
g^\star(\nabla_V\bT, U)+g^\star(V,\nabla_U\bT)=-ab'+a'b-a'b+ab'=0. 
$$
Finally, it follows from equations (\ref{eq:CT}) that $R(U,\bT)V=-ac'\bX-bc'\bY+(aa'+bb')\bT$. 
Because $g^\star(U, V)\bT-g^\star(\bT, V)U=(aa'+bb'+cc')\bT-c'(a\bX+b\bY+c\bT)$, we also have
$$R(U,\bT)V=g^\star(U, V)\bT-g^\star(\bT,V)U$$ for any vector fields $U, V$ on $\frakH^{\star}$. Now the theorem is proved.
\end{proof}
 Using the relation $K(U, V)=g(R(U, V) U, V)$ for sectional curvature of planes spanned by unit vectors $U, V$, we obtain the following:
\begin{cor}\label{cor:kurv}
The sectional curvatures of distinguished planes are:
$$
K(\bX,\bY)=-7,\quad K(\bX,\bT)=1,\quad K(\bY,\bT)=1.
$$
\end{cor}
Recall that if $\{X_1,X_2,X_3\}$ is an orthonormal  basis of a Riemannian 3-manifold $M$, then the Ricci curvature in the direction of $X_i$ is
$$
{\rm Ric}(X_i)=\frac{1}{2}\sum_{j\neq i}K(X_i,X_j).
$$
Moreover, the scalar curvature $K$ is
$$
K=\frac{1}{3}\sum_{i=1}^3{\rm Ric}(X_i).
$$
We obtain straightforwardly:
\begin{cor}\label{cor:RicH}
The Ricci curvatures ${\rm Ric}(\bX)$, ${\rm Ric}(\bY)$, ${\rm Ric}(\bT)$, in the directions of $\bX$, $\bY$ and $\bT$  are respectively
\begin{eqnarray*}
&&
{\rm Ric}(\bX)=%\frac{1}{2}\left(K(X,Y)+K(X, T)\right)=
-3,\quad
%&&
{\rm Ric}(\bY)=%\frac{1}{2}\left(K(Y, X)+K(Y, T)\right)=
-3,\quad
%&&
{\rm Ric}(\bT)=%\frac{1}{2}\left(K(T, Y)+K(T, X)\right)=
1.
\end{eqnarray*}
Therefore the scalar curvature is $K=-\frac{5}{3}$.
\end{cor}

\subsubsection{From $\frakH^{\star}$ to $T_1({\bH^{1}_\C})$}\label{sec:T1}

There is a number of geometric objects to which $\frakH^\star$ may be identified to. In particular, $\frakH^\star$ can be identified to the unit tangent bundle $T_1(\bH^1_\C)$ of the hyperbolic space $\bH^1_\C$. To see this, we consider the hyperbolic plane $\bH^1_\C$ modelled in the left half-plane $\calL=\{\zeta\in\C\;|\;\Re(\zeta)<0\}$ and its unit tangent bundle  $T_1(\bH^1_\C)$, 
which is diffeomorphic to $\bH^1_\C\times S^1$.
There is  Kor\'anyi map $K:\frakH^\star\to T_1(\bH^1_\C)$  given by
%{\textcolor{red}{$K(z,t)= (-|z|^2+it,\; e^{i \arg z}),\quad(z,t)\in\frakH^\star.$ refer to[12] and Section 3.3.1}}
\begin{equation}\label{eq:Kmap}
K(z,t)= (-|z|^2+it,\;\arg z),\quad(z,t)\in\frakH^\star.
\end{equation}
The inverse $K^{-1}:T_1(\bH^1_\C)\to\frakH^\star$ is then given by 
\begin{equation}\label{eq:kormapinv}
K^{-1}(\zeta, \phi)=\left(\sqrt{-\Re(\zeta)}e^{i\phi},\; \Im(\zeta)\right),\quad (\zeta, \phi)\in T_1(\bH^1_\C).
\end{equation}
and the map $K$ is a diffeomorphism.

Let $(\frakH^\star;\omega^\star,\bT,\Phi^\star,g^\star)$ be the Sasakian structure of the 
affine-rotational group. 
%and consider the the Kor\'anyi map $K$ and its inverse $K^{-1}$ (Eqs. (\ref{eq:Kmap}) and (\ref{eq:kormapinv}), respectively). 
Suppose that $(\xi+i\eta, \phi)$ are coordinates for $\bH^1_\C\times S^1$. Then let
$$\omega=(K^{-1})^\ast \omega^{\star},\quad T=K_{*}(\bT),\quad \Phi=K_{\ast}\circ\Phi^{\star}\circ(K^{-1})_\ast,
\quad g=(K^{-1})^{\ast}g^{\star}.$$ Explicitly,
$$\omega=-\frac{d\eta}{2\xi}+d\phi, \quad T=\partial_\phi,\quad
g=\frac{d\xi^2+d\eta^2}{4\xi^2}+\left(d\phi-\frac{d\eta}{2\xi}\right)^2.$$
 
One can know that
$(\bH^1_\C\times S^1; \omega, T, \Phi, g)$ is a Sasakian structure for the unit tangent bundle of the hyperbolic plane, see \cite{Sa}.
 The Carnot-Carath\'eodory isometry group is just $\rm{SU}(1, 1)$.  
 %{\color{blue} we refer to Theorem 10 in \cite{Sa}.}

\subsection{The K\"ahler structure of $\calC(\frakH^\star)$}\label{sec:kcon}
Since $(\frakH^\star;\omega^\star,\bT,\Phi^\star,g^\star)$ is Sasakian we immediately have for the Riemannian cone $\calC(\frakH^\star)=\frakH^\star\times_{r^2}\R_{>0}$ is K\"ahler. Here, the complex structure of $\frakH^\star\times\R_{>0}$ is given in terms of the basis $\{\bX,\bY,\bT,r\partial_r\}$ by
\begin{equation}\label{eq:J-L}
\J\bX=\bY,\quad \J\bY=-\bX,\quad \J\bT=-r\partial_r,\quad \J(r\partial_r)=\bT,
\end{equation}
and the K\"ahler metric as well as the fundamental 2-form are, respectively,

\begin{eqnarray}
&&\label{eq:grHC}
g_r^\star=dr^2+r^2\;g^\star,\\
&&\label{eq:OrHC}
\Omega_r^\star=d\left(\frac{r^2\omega^\star}{2}\right)=r\;dr\wedge\omega^\star+\frac{r^2}{2}d\omega^\star=r\;dr\wedge\omega^\star+r^2\;\phi^\star\wedge\psi^\star.
\end{eqnarray}

\medskip

It is clear that $\phi^r=r\phi^\star$, $\psi^r=r\psi^\star$, $\omega^r=r\omega^\star$ and $dr$ form an orthonormal basis for the cotangent space of $\calC(\frakH^\star)$; in this basis
$$
\Omega_r^\star=\phi^r\wedge\psi^r+dr\wedge\omega^r.
$$
The dual basis is the set $\{\bX_r,\bY_r,\bT_r,\bS_r\}$ where
$$
\bX_r=(1/r)\bX,\quad \bY_r=(1/r)\bY,\quad \bT_r=(1/r)\bT,\quad \bS_r=\partial/\partial r.
$$
The only non-vanishing Lie bracket relations are
\begin{eqnarray*}
&&
[\bX_r,\bY_r]=(2/r)(\bY_r-\bT_r),\quad[\bX_r,\bS_r]=(1/r)\bX_r,\quad
[\bY_r,\bS_r]=(1/r)\bY_r,\quad[\bT_r,\bS_r]=(1/r)\bT_r.
\end{eqnarray*}

A basis for the (1, 0) tangent space comprises $\bZ, \bW$, where
$$\bZ=\frac{1}{2}(\bX-i\bY),\qquad \bW=\frac{1}{2}(\bT+i\bS),$$
where $\bS=r\partial_r$.
 Accordingly, the (1, 0) cotangent space has a basis comprising 
$d \bZ, d\bW$, where
$$d\bZ=\phi^\star+i \psi^\star,\qquad d\bW=\omega^{\star}-i(1/r)dr.$$
%so that $$\Omega_r^{\star}=\frac{i}{2}\left(d\bZ_r\wedge d\overline{\bZ_r}+d \bW_r\wedge d\overline \bW_r\right).$$
Let $\rho: \calC(\frakH^{\star})\to \R$ be a smooth function. Then the (1, 0) and (0, 1) differentials are given respectively by
$$\partial \rho=\bZ(\rho)d\bZ+\bW(\rho)d\bW,$$
$$\bar{\partial}\rho=\overline{\bZ}(\rho)d\overline{\bZ}+\overline{\bW}(\rho)d\overline{\bW}.$$

\medskip

In the next proposition we compute the sectional curvatures of distinguished planes. 
\begin{prop}
All sectional curvatures of distinguished planes vanish besides that of the distinguished two planes spanned by $\bX_r, \bY_r$ and $\bT_r, \bS_r$ respectively:
$$
K_r(\bX_r, \bY_r)=-8/r^2<0, \qquad K_r(\bT_r, \bS_r)=-1/r^2<0.
$$
\end{prop}
\begin{proof}
If $\nabla^r$ is the Riemannian connection, we obtain 
\begin{eqnarray*}
&&
\nabla^r_{\bX_r}\bX_r=-(1/r)\bS_r,\quad \nabla^r_{\bY_r}\bX_r=(1/r)(-2\bY_r+\bT_r),\quad \nabla^r_{\bT_r}\bX_r=(1/r)\bY_r,\quad\nabla^r_{\bS_r}\bX_r=0,\\
&&
\nabla^r_{\bX_r}\bY_r=-(1/r)\bT_r,\quad \nabla^r_{\bY_r}\bY_r=(1/r)(2\bX_r-\bS_r),\quad \nabla^r_{\bT_r}\bY_r=-(1/r)\bX_r,\quad\nabla^r_{\bS_r}\bY_r=0,\\
&&
\nabla^r_{\bX_r}\bT_r=(1/r)\bY_r,\quad \nabla^r_{\bY_r}\bT_r=-(1/r)\bX_r,\quad \nabla^r_{\bT_r}\bT_r=-(1/r)\bS_r,\quad\nabla^r_{\bS_r}\bT_r=0,\\
&&
\nabla_{\bX_r}\bS_r=(1/r)\bX_r,\quad\nabla^r_{\bY_r}\bS_r=(1/r)\bY_r,\quad\nabla^r_{\bT_r}\bS_r=(1/r)\bT_r,\quad\nabla^r_{\bS_r}\bS_r=0.
\end{eqnarray*}
Hence for the Riemannian curvature tensor $R^r$ we have
\begin{eqnarray*}
&&
R^r(\bX_r,\bY_r)\bX_r=-(8/r^2)\bY_r,\qquad
R^r(\bT_r,\bS_r)\bT_r=-(1/r^2)\bS_r.
\end{eqnarray*}
whereas
\begin{eqnarray*}
&&
R^r(\bX_r,\bT_r)\bX_r=0,\quad
R^r(\bX_r,\bS_r)\bX_r=0,\\
&&
R^r(\bY_r,\bT_r)\bY_r=0,\quad
R^r(\bY_r,\bS_r)\bY_r=0,\quad
\end{eqnarray*}
The proof follows.
\end{proof}
\begin{cor}
The Ricci curvatures of $g_r$ in the directions of $\bX_r,\bY_r,\bT_r$ and $d/dr$ are respectively
\begin{equation*}
{\rm Ric}(\bX_r)={\rm Ric}(\bY_r)=-\frac{8}{3r^2},\quad {\rm Ric}(\bT_r)={\rm Ric}(\bS_r)=-\frac{1}{3r^2},
\end{equation*} 
and the scalar curvature is
\begin{equation*}
K=-\frac{3}{2r^2}.
\end{equation*}
\end{cor}

\medskip
Note that the complex coordinates of the K\"ahler cone $\calC(\frakH^\star)$ is $(z, t+i\log r).$ The affine-rotational group $\frakH^\star$ is embedded into $\calC(\frakH^\star)$ as the hypersurface $r=1$.
%\begin{prop}\label{prop:HHHKstar}
%The affine-rotational group $\frakH^\star$ is embedded into $\calC(\frakH^\star)$ as the hypersurface $r=1$. %If $\calH^\star$ is the {\rm CR} structure of $\frakH^\star$ and $\calH^\star_r=\ker(dr+ir\omega^\star)$, then the embedding is $(\calH^\star,\calH^\star_r)$-K\"ahler.
%\end{prop}

\section{Configuration Space of Four Points in $\partial\HC$}\label{sec-conf}

We let $\frakC_4$ be the set of ordered quadruples of pairwise distinct points in the boundary of the complex hyperbolic plane $\partial\HC$ and we denote by $\frakF_4$  the configuration space of $\frakC_4$, that is, the quotient of $\frakC_4$ with respect to the diagonal action of $\PU(2, 1)$. There are certain subsets of $\frakF_4$ with interesting geometrical properties; those properties have been studied in \cite{FP}:
\begin{itemize}
\item The subset $\frakF_4^\R$ comprising orbits of quadruples $\fp=(p_1,p_2,p_3,p_4)$ such that not all $p_i$ lie in the same $\C$-circle. This is a 4-dimensional real manifold. 
\item The subset $\frakF_4'\subset\frakF_4^\R$  comprising orbits of quadruples $\fp=(p_1,p_2,p_3,p_4)$ such that $p_2,p_3,p_4$ do not lie in the same $\C$-circle. 
This is a CR manifold of codimension 2, see Section \ref{sec-CRV4} below.
\item The subset $\frakF^\C_4\subset\frakF_4^\R$  comprising orbits of quadruples $\fp=(p_1,p_2,p_3,p_4)$ such that $p_1,p_4$ do not lie in the same orbit of the stabiliser of $p_2,p_3$. This is a 2-dimensional (disconnected) complex manifold, see Section \ref{sec-complex}. 
\end{itemize}
We introduce the subset $\frakF_4''$
comprising orbits of quadruples $\fp=(p_1,p_2,p_3,p_4)$ such that
\begin{itemize}
\item
 $p_1,p_2,p_3$ as well as $p_2,p_3,p_4$ do not all lie in the same $\C$-circle.
\item $p_1$ and $p_4$ are not in the same orbit of the stabiliser of $p_2$ and $p_3$.
\end{itemize}
It is clear that $\frakF_4''\subset\frakF_4^\C$.

\medskip

In \cite{FP}, see also \cite{Pla-CR}, the configuration space was identified to the cross-ratio variety $\frakX$. For clarity, we recall basic facts about $\frakX$ in Section \ref{sec-X}.
We prove in Section \ref{sec-mainth} our main Theorem \ref{thm-main}: the set $\frakF''_4$ can be endowed with a K\"ahler structure.  In Section \ref{sec-var} we then follow a slightly different route from the result of Gusevskii-Cunha (please refer to \cite{GC}) to show that: the set $\frakF''_4$ is naturally identified to a 4-dimensional variety inside $\C^2\times(-\pi/2,\pi/2)$. We denote this variety by $\frakV_4$ and we recover the codimension 2 CR structure $\calH$ of $\frakF''_4$ in Section \ref{sec-CRV4}. The CR structure $\calH$ is generated by a $(1,0)$ vector field $Z$, see (\ref{eq-CRV}). Recall now that the subbundle $\calH^\star$ of $T^{(1,0)}(\calC'(\frakH^\star))$ is generated by the $(1,0)$ vector field $\bZ$.  In Section \ref{sec-CReq} we prove that there is a diffeomorphism from $\calC'(\frakH^\star)$ to $\frakV_4$, which is CR with respect to $\calH$ and $\calH^\star$. In Section \ref{sec-complex} we  consider two complex structures of $\frakF''_4$ and show that they are CR equivalent but not biholomorphic.

\medskip

\subsection{Invariants, Cross-Ratio Variety}\label{sec-X}
Recall that the {\it Cartan's angular invariant} $\mathbb{A}(\fp)$ of an ordered triple $\fp=(p_1,p_2,p_3)$ of pairwise distinct points in $\partial\bH^2_\C$ is defined by 
$$
\mathbb{A}(\fp)=\arg\left(-\langle \bp_1,\bp_2\rangle\langle \bp_2,\bp_3\rangle\langle \bp_3,\bp_1\rangle\right)\in[-\pi/2,\pi/2],
$$
where $\bp_i$ are the lifts of $p_i$ respectively. The definition is independent of the choice of lifts.
%, and remains invariant under the action of ${\rm PU}(2,1)$. 
Cartan's angular invariant $\mathbb{A}(\fp)$ satisfies the properties (see \cite{G}): $\A(\fp)=\pm\pi/2$ if and only if $\fp$ is a triple of points lying in the same $\C$-circle and $\A(\fp)=0$ if and only if $\fp$ is a triple of points lying in the same $\R$-circle; moreover, for triples $\fp$ and $\fp'$ of points not lying in the same $\C$-circle, there exists a $g\in{\rm PU}(2,1)$ such that $g(\fp)=\fp'$ (that is, $g(p_i)=p_i'$, $i=1,2,3$) if and only if $\A(\fp)=\A(\fp')$. 

Given a quadruple $\fp=(p_{1},p_{2}, p_{3},p_{4})\in\frakC_4$, then its {\it cross-ratio} is defined by
\[
\X(\fp)=\mathbb{X}(p_{1},p_{2}, p_{3},p_{4})=\frac{\langle \bp_{4},\bp_{2}\rangle\langle \bp_{3},\bp_{1}\rangle}{\langle \bp_{4},\bp_{1}\rangle\langle \bp_{3},\bp_{2}\rangle},
\]
where $\bp_i$ are lifts of $p_i$. The cross-ratio is independent of the choice of lifts  and remains invariant under the action of ${\rm PU}(2,1)$. 
By permuting the points of the configuration $\fp$  we will obtain 24 cross-ratios; all these are functions of the following three cross-ratios:
$$
\X_1=\X(p_1,p_2,p_3,p_4),\quad
\X_2=\X(p_1,p_3,p_2,p_4),\quad
\X_3=\X(p_2,p_3,p_1,p_4).
$$
These three cross-ratios satisfy the following two real equations:
\begin{equation}\label{eq-X}
\begin{aligned}
&
|\X_2|=|\X_1|\cdot|\X_3|,\\
&
|\X_1+\X_2-1|^2=2\Re\left(\X_1(\overline{\X_2}+\overline{\X_1}\X_3)\right).
\end{aligned}
\end{equation}
Equations (\ref{eq-X}) define the cross-ratio variety $\mathfrak{X}$, see \cite{FP}.  The following identifications hold:
\begin{itemize}
\item The subset $\frakF_4^\R$ is identified to the subset $\frakX\setminus\frakX_\R$ of $\frak{X}$, where $\frakX_\R$ is the subset of $\frakX$ comprising $(\X_1,\X_2,\X_3)$ such that all $\X_i$ are real, $\X_1+\X_2=1$ and $\X_3=1-(1/\X_1)$.
\item The subset $\frakF_4'$ is identified to the subset $\frakX\setminus\frakX'$ of $\frak{X}$, where $\frakX'$ is the subset of $\frakX$ comprising $(\X_1,\X_2,\X_3)$ such that 
$$
\X_1+\X_2=1,\quad \X_3+\frac{1}{\X_1}=1,\quad\frac{1}{\X_2}+\frac{1}{\X_3}=1.
$$ 
\item The subset $\frakF^\C_4$ is identified to the subset $\frakX\setminus\frakX_\C$ of $\frak{X}$, where $\frakX_\C$  is the subset of $\frakX$ comprising $(\X_1,\X_2,\X_3)$ such that $\X_3\in\R$.
\end{itemize}

\subsection{Main Theorem: $\frakF_4''$ and $\mathcal C'(\frakH^{\star})$}\label{sec-mainth}
We let
$$
\calC'(\frakH^{\star})=\calC(\frakH^{\star})\setminus\{(z,t,r)\;|\;\log r=t/|z|^2\}.
$$
In this section, we shall prove our main theorem:
\begin{thm}\label{thm-main}
There is a bijection $\mathcal B_0: \frakF_4''\to\mathcal C'(\frakH^{\star})$. Therefore $\frakF_4''$ inherits the K\"ahler structure of 
$\mathcal C(\frakH^{\star})$. 
\end{thm}
\begin{proof}
By the doubly transitive action on $\partial\HC$ of $\PU(2, 1)$, we normalise a quadruple $\fp=(p_1, p_2, p_3, p_4)\in\frakC''_4$ such that 
$p_2=\infty, \, p_3=0.$ Assuming that $\A(p_1, p_2, p_3)=a,$ one can get that $a\in (-\pi/2, \pi/2)$ and
$$p_1=(z_1, |z_1| \tan a), z_1\in\C_\ast.$$ Because each element in the stabiliser of $\infty, 0$ is of the
Heisenberg translation form $$(z, t)\mapsto(\lambda z, |\lambda|^2t), \, \lambda\in\C_\ast,$$ we can finally normalise $\fp=(p_1, p_2, p_3, p_4)\in\frakC_4''$ to be 
$$p_1=(1, \tan a),\quad  p_2=\infty, \quad  p_3=0, \quad  p_4=(z , t),$$
where $a\in(-\pi/2, \pi/2),$ $z\neq 0, (z, t)\neq (\lambda, |\lambda|^2\tan a)$, for each $\lambda\in\C_*$.
We then consider $\widetilde{\mathcal B_0}: \frakC_4''\to\mathcal C'(\frakH^{\star})$ given by 
$$\widetilde{\mathcal B_0}(\fp) = (z, t, e^{\tan a}).$$
Then the followings hold:
\begin{enumerate}
\item[{(i)}] If $(z, t, r)\in \mathcal C'(\frakH^{\star})$, then there exists a $\fp\in\frakC_4''$ such that $\widetilde{\mathcal B_0}(\fp) = (z, t, r).$ 
Indeed we may consider $\fp$ with 
$$p_1=(1, \log r),\quad  p_2=\infty, \quad  p_3=0, \quad  p_4=(z , t).$$
\item[{(ii)}] If $\fp\in\frakC_4''$ and  {{$g\in\rm{SU}(2, 1)$, }}then $\widetilde{\mathcal B_0}(\fp)=\widetilde{\mathcal B_0}(g(\fp)).$
\item[{(iii)}] If $\widetilde{\mathcal B_0}(\fp)=\widetilde{\mathcal B_0}({\fp^{\prime}})$ for $\fp, \fp'\in\frakC_4''$, 
then there exists a 
$g\in\rm{SU}(2, 1)$ such that 
$\fp'=g(\fp).$ To see this, 
we could normalise so that 
\begin{align*}
&p_1=(1, \tan a),\quad\,  p_2=\infty, \quad  p_3=0, \quad  p_4=(z , t),\\
&p_1'=(1, \tan a'),\quad  p_2'=\infty, \quad   p_3'=0, \quad  p_4'=(z' , t').
\end{align*}
A $g\in\rm{SU}(2, 1)$ mapping $\fp$ to $\fp'$ must be of the form $\rm{E}_\lambda\in\rm{SU}(2, 1)$, $\lambda=l+i\theta\in\C_{\ast}$, that is,
$$E_\lambda(z, t)=(e^{l+i\theta}z, e^{2l}t),$$
since it belongs to $\rm{Stab}(0, \infty).$ 
It is now clear that $p_1'=p_1,$ and $p_4'=p_4.$ 
\end{enumerate}
 The proof is complete.
\end{proof}
%\begin{rem}
%It is also interesting to consider the space of quadruples $\fp=(p_1,p_2,p_3,p_4)$ such that $p_1, p_2, p_3$ do not lie in the same  $\mathbb{C}$-circle. We denote its quotient with respect to the diagonal action of $\PU(2, 1)$ by $\mathcal{M}_4$. Note that under this condition $p_1, p_2, p_3$ are pairwise distinct, otherwise they will lie in a same $\mathbb{C}$-circle; but $p_4$ could be one of them. It is a degenerate case of $\frakF_4'$. One can get that the bijection of $\mathcal{M}_4$ with $\partial\HC\times\mathbb{R}_{+}$ from the similar proof as above. Then It follows from that $\mathcal{M}_4$ is K\"ahler. Therefore its subspaces $\frakF_4'$ and $\frakF_4''$ are both K\"ahler.
%\end{rem}

\subsection{$\frakF_4''$ and the variety $\frakV_4$}\label{sec-var}
The subset $\frakF_4''$ can be identified to the variety $\frakV_4$ of $\C^2\times(-\pi/2,\pi/2)$ which we will define below. In the first place, we show that if $[\fp]\in\frakF_4''$, $\fp=(p_1,p_2,p_3,p_4)$ and   $(\X_1,\X_2,\A)$ correspond to $\fp$, that is, $\X_1=\X_1(\fp)$ $\X_2=\X_2(\fp)$ and $\A=\A(p_1,p_2,p_3)$, then 
\begin{equation}\label{crv}
|\X_1+\X_2-1|^2=2\Re\left(\X_1\overline{\X_2}(1+e^{-2i\A})\right),
\end{equation}
where $\X_1+\X_2-1\neq 0$, $\Re\left(\X_1\overline{\X_2}e^{-i\A}\right)>0$ and $\arg(\X_1/\X_2)\neq 2a$. Note that we always have
\begin{equation}
|\X_1+\X_2-1|^2\le 2\Re(\X_1\overline{\X_2})+2|\X_1||\X_2|,
\end{equation}
with the equation holding only if $\arg(\X_1/\X_2)=2a$.

A variation of formula (\ref{crv}) is found in \cite{CG}. For completeness, we prove (\ref{crv}) here in a different way; we mention that 
equation (\ref{crv})
follows in \cite{CG} from the vanishing of the determinant of the Gram matrix of points of $\fp$. In our setting we make no use of Gram determinants. 

We may normalise so that $p_1=(1,\tan a)$, $p_2=\infty$, $p_3=0$ and $p_4=(z,t)$ with lifts 
$$
\bp_1=\left[\begin{matrix} -1+i \tan a\\ \sqrt{2}\\1\end{matrix}\right],\quad
\bp_2=\left[\begin{matrix} 1\\0\\0\end{matrix}\right],\quad
\bp_3=\left[\begin{matrix} 0\\ 0\\1\end{matrix}\right],\quad
\bp_{4}=\left[\begin{matrix} u\\\sqrt{2}\;z\\1\end{matrix}\right].
$$
Here $a=\A\in(-\pi/2,\pi/2)$, $z\neq 0$, $t/|z|^2\neq \tan a$ and $u=-|z|^2+it$.  
One can calculate directly that 
\begin{eqnarray*}
&\X_1&=\frac{-1-i\tan a}{u-1-i\tan a+2z},\\
&\X_2&=\frac{u}{u-1-i\tan a+2z}.
\end{eqnarray*}
Note that $\X_1,\X_2$ are well defined (the denominator does not vanish) since $(z,t)\neq(1,\tan a)$. Equation (\ref{crv}) now follows by straightforward calculation. 
%Since
%$$z=\frac{\X_1+\X_2-1}{(1+e^{-2ia})\X_1}$$
%we also have $\X_1+\X_2-1\neq 0$. 
From
$$\X_1+\X_2-1=-\frac{2z}{u-1-i\tan a+2 z},$$
we get that $\X_1+\X_2-1\neq0$ because of $z\neq0.$
On the other hand, we have
$$
\Re\left(\X_1\overline{\X_2}e^{-ia}\right)=\frac{|z|^2}{\cos a|u+2z-1-i\tan a|^2}>0.
$$
Finally,  one can get that $\arg(-\bar{u})\neq a$, otherwise $p_1,p_4$ are in the same orbit of the stabiliser of $p_2,p_3$. This gives $\arg(\X_1/\X_2)\neq 2a$.%=\arg{\X_1\overline{\X_2}}

\bigskip

We now define $\frakV_4$ to be the subset of $\C^2\times(-\pi/2,\pi/2)$ comprising $(w_1,w_2,a)$ such that
$$
|w_1+w_2-1|^2=2\Re\left(w_1\overline{w_2}(1+e^{-2ia})\right),%\quad ,\;.
$$
with $w_1+w_2-1\neq 0$, $\Re\left(w_1\overline{w_2}e^{-ia}\right)>0$ and $\arg(w_1/w_2)\neq 2a$.

The identification of $\frakF''_4$ to $\frakV_4$ is by the mapping $B_0:\frakF_4''\to\frakV_4$ given for each $[\fp]$ by
\begin{equation}
B_0([\fp])=\left(\X_1,\X_2,\A(\fp)\right),
\end{equation}
where %each orbit $[\fp]$ to $(\X_1,\X_2,\A)$, where $\X_i=\X_i(\fp)$, $i=1,2$ and 
$\A(\fp)=\A(p_1,p_2,p_3)$. This mapping is bijective. Indeed, given $(w_1,w_2,a)\in\frakV_4$ we may consider the quadruple $\fp$ of points $p_1=(1,\tan a)$, $p_2=\infty$, $p_3=o$ and $p_4=(z,t)$ where
$$
z=\frac{w_1+w_2-1}{(1+e^{-2ia})w_1},\quad t=-2\Im\left(\frac{w_2}{(1+e^{-2ia})w_1}\right).
$$
Straightforward calculation yields to
$$
\X_1=w_1,\quad \X_2=w_2,\quad \A(p)=a.
$$ 
If additionally, for $\fp, \fp'\in\frakC_4''$ there exists a $g\in{\rm SU}(2,1)$ such that $g(\fp)=\fp'$, 
%(i.e., $g(p_i)=p_i'$ for $i=1,\dots,4$), 
then clearly $\X_i=\X_i'$, $i=1,2$ and $\A(\fp)=\A(\fp')$. Finally, if $\fp, \fp'\in\frakC_4''$ such that $\X_i=\X_i'$, $i=1,2$ and $\A(\fp)=\A(\fp')$, then we may normalize so that 
%$p_1=(1,\tan a)$, $p_2=\infty$, $p_3=o$, $p_4=(z,t)$ and $p_1'=(1,\tan a')$, $p_2'=\infty$, $p_3'=o$ and $p_4'=(z',t')$. 
\begin{align*}
&p_1=(1, \tan a),\quad\,  p_2=\infty, \quad  p_3=0, \quad  p_4=(z , t),\\
&p_1'=(1, \tan a'),\quad  p_2'=\infty, \quad   p_3'=0, \quad  p_4'=(z' , t').
\end{align*}
Due to our assertion, we then obtain that $p_1=p_1'$ and $p_4=p_4'$.
 
%\begin{rem}
%We remark that there is also a bijection between varieties $\frakX\setminus\frakX''$ and $\frakV_4$; this is given by
%$$
%\frakV_4\ni(\X_1,\X_2,\A)\mapsto\left(\X_1,\X_2,\frac{\X_2}{\X_1}e^{2i\A}\right)\in \frakX\setminus\frakX''.
%$$
%\end{rem}
\subsection{CR structure of $\frakF_4''$}\label{sec-CRV4}
We now describe the CR structure of $\frakF_4''$ as this is obtained by its identification to $\frakV_4$. For this, we consider the defining function of $\frakV_4$:
$$
F(w_1,w_2,a)=|w_1+w_2-1|^2-2\Re(w_1\overline{w_2}(1+e^{-2ia}))=0.
%F(w_1,w_2,a)=|w_1+1-w_2|^2-2\Re(w_1(1+e^{-2ia}))=0.
$$
The CR structure is obtained as the kernel of
\begin{eqnarray*}
\partial F&=&\frac{\partial F}{\partial w_1}dw_1+\frac{\partial F}{\partial w_2}dw_2,
%\frac{\partial F}{\partial z}dz+\frac{\partial F}{\partial w}dw;
\end{eqnarray*}
where
\begin{eqnarray*}
\frac{\partial F}{\partial w_1}&=&%\overline{w_1}-\overline{w_2}-e^{-2ia}
\overline{w_1}-e^{-2ia}\overline{w_2}-1=-\beta,\\
\frac{\partial F}{\partial w_2}&=&%-\overline{w_1}-1+\overline{w_2}
\overline{w_2}-e^{2ia}\overline{w_1}-1=\alpha.
\end{eqnarray*}
Note that $\alpha$ and $\beta$ cannot be simultaneously zero; this would lead to the contradiction $a=\pm\pi/2$. Therefore the codimension 2 CR structure on $\frakV_4$ is 
\begin{equation}\label{eq-CRV}
\calH=\ker(\partial F)=\left\langle Z=\alpha\frac{\partial}{\partial w_1}+\beta\frac{\partial}{\partial w_2}\right\rangle.
\end{equation}
Since
$$
\partial\overline{\partial}F=dw_1\wedge d\overline{w_1}
-e^{2ia}dw_2\wedge d\overline{w_1}-e^{-2ia}dw_1\wedge d\overline{w_2}+
dw_2\wedge d\overline{w_2},
$$
we obtain
$$
\partial\overline{\partial}F(Z,\overline{Z})=|\alpha|^2-2\Re(\alpha\overline{\beta}e^{-2ia})+|\beta|^2=|1+e^{2ia}|^2=4\cos^2a>0,
$$
that is, the CR structure is strictly pseudoconvex. Set $w_i=u_i+iv_i$, $i=1,2$, $\alpha=\alpha_1+i\alpha_2$ and $\beta=\beta_1+i\beta_2$. Then the contact form $\tau$ is
$$
\tau=-\beta_2du_1-\beta_1dv_1+\alpha_2du_2+\alpha_1dv_2.
$$ 
\subsection{CR-equivalence}\label{sec-CReq}
%\subsection{Main Theorem: $\frakF_4''$ and $\mathcal C(\frakH^{\star})$}\label{sec-mainth}
%In this section, we shall prove our main theorem:
%\begin{thm}\label{thm-main}
%The set $\frakF_4''$ may be endowed with the K\"ahler structure of 
%$\mathcal C(\frakH^{\star})$. 
%\end{thm}
%\begin{proof}
%It suffices to establish a diffeomorphism $G:\calC(\frakH^\star)\to\frakV_4$. For this, l
Let $\calC'(\frakH^\star)$ with coordinates $(z,t,r)$ and the variety $\frakV_4$ with coordinates $(w_1,w_2,a)$. We also set $u=-|z|^2+it$ and $q=-1-i\log r$. Then we consider the mapping $G:\calC'(\frakH^\star)\to\frakV_4$, where
\begin{equation*}
G(z,t,r)=\left(\frac{q}{u+2z+q},\;\frac{u}{u+2z+q},\;\arctan(-\Im(q))\right).
\end{equation*}
This mapping is a diffeomorphic bijection; the inverse $G^{-1}:\frakV_4\to\calC'(\frakH^\star)$ is given by
\begin{equation*}
G^{-1}(w_1,w_2,a)=\left(\frac{w_1+w_2-1}{(1+e^{-2ia})w_1},\; 
-2\Im \left(\frac{w_2}{(1+e^{-2ia})w_1}\right),\;e^{\tan a} %\arctan\log r
\right).
\end{equation*}
%\end{proof}
Using the map $G$ we shall display the equivalence of CR structures for $\frakF_4''.$
\begin{prop}\label{prop-CR-H}
We consider $\calC'(\frakH^\star)$ with the CR structure $\calH^\star=\langle\bZ\rangle$ and the variety $\frakV_4$ with the CR structure $\calH$ as in (\ref{eq-CRV}). Then the  diffeomorphism $G:\calC'(\frakH^{\star})\to\frakV_4$ is CR with respect to these structures. 
\end{prop}
\begin{proof}
We calculate directly:
\begin{eqnarray*}
%G_*(\bZ_r)&=&\bZ_r\left(\frac{-1-i \log r}{-|z|^2+2z+i t-1-i \log r}\right)\frac{\partial}{\partial w_1}+\bZ_r\left(\frac{-1+i \log r}{-|z|^2+2\bar{z}-i t-1+i \log r}\right)\frac{\partial}{\partial \overline{w_1}}+\\
G_*(\bZ)&=&\bZ\left(\frac{q}{u+2z+q}\right)\frac{\partial}{\partial w_1}+\bZ\left(\frac{\overline{q}}{\overline{u}+2\bar{z}+\bar{q}}\right)\frac{\partial}{\partial \overline{w_1}}+\\
& &\bZ\left(\frac{u}{u+2z+q}\right)\frac{\partial}{\partial w_2}+\bZ \left(\frac{\bar{u}}{\bar{u}+2\bar{z}+\bar{q}}\right)\frac{\partial}{\partial \overline{w_2}}\\
%& &\bZ_r \left(\frac{-|z|^2+i t}{-|z|^2+2z+i t-1-i \log r}\right)\frac{\partial}{\partial w_2}+\bZ_r \left(\frac{-|z|^2-i t}{-|z|^2+2\bar{z}-i t-1+i \log r}\right)\frac{\partial}{\partial \overline{w_2}}\\
&& +\bZ(\arctan\log r)\frac{\partial}{\partial a}.
\end{eqnarray*}
The second and the fourth term vanish since $\bZ(\bar{u})=\bZ(\bar{z})=0$. The last term also vanishes because $\bZ_r$ does not involve derivation with respect to $r$. On the other hand, since
$$
\bZ(u)=-{2|z|^2},\quad \bZ(z)={z},
$$
we have
$$
\bZ\left(\frac{q}{u+2z+q}\right)={2z}\cdot\frac{q(\bar{z}-1)}{(u+2z+q)^2},\quad
\bZ\left(\frac{u}{u+2z+q}\right)={2z}\cdot\frac{\bar{u}-\bar{z}q}{(u+2z+q)^2}.
$$
Therefore we get that
\begin{equation*}
G_*(\bZ)={2z}\cdot\frac{1}{(u+2z+q)^2}\left(q(\bar{z}-1)\frac{\partial}{\partial w_1}+(\bar{u}-\bar{z}q)\frac{\partial}{\partial w_2}\right).
\end{equation*}
By calculating
\begin{eqnarray*}
&&
{2z}\cdot\frac{1}{(u+2z+q)^2}=\frac{1}{2}\cdot w_1(w_1+w_2-1)(1+e^{-2ia}),\\
&&
q(\bar{z}-1)=-\frac{2(\overline{w_2}-e^{2ia}\overline{w_1}-1)}{|1+e^{-2ia}|^2\overline{w_1}},\\
&&\bar{u}-\bar{z}q=\frac{2(\bar{w_1}-\bar{w_2}e^{-2ia}-1)}{|1+e^{-2ia}|^2\overline{w_1}},
\end{eqnarray*}
we finally obtain that
$$
G_*(\bZ)=kZ,\quad k=-\frac{w_1(w_1+w_2-1)}{(1+e^{2ia})\bar{w_1}},
$$
where $Z$ is as in (\ref{eq-CRV}). This proves the result.
\end{proof}

\subsection{Comparison of complex structures}\label{sec-complex}
For the definition of the complex structure in the subset $\frakF^\C_4$ we refer to \cite{FP}, \cite{Pla-CR}. In brief, the complex structure is obtained by identifying the set $\frakF^\C_4$ to $\C P^1\times\C\setminus\R$; the map $\calB_1$ defined in Theorem \ref{thm-comp} below is actually the restriction of this identification in $\frakF''_4$. %The subset $\frakF^\C_4\cap\frakF_4'$ is the one that will concern us in this section. 

\medskip

\begin{thm}\label{thm-comp}
The set $\frakF''_4$ admits two complex structures. The identity map of $\frakF''_4$ is CR but not biholomorphic with respect to these two complex structures. 
%Let $\calC'(\frakH^\star)$ be the (open and disconnected) subset of  $\calC(\frakH^\star)$ at which
%$$
%\log r-\frac{t}{|z|^2}\neq 0.
%$$ 
%; therefore the set $\frakF''_4$ admits two complex structures, namely the one from $\C_*\times(\C\setminus\R)$ and the other from $\calC'(\frakH^\star)$. The identity map of $\frakF''_4$ is CR with respect to these two complex structures.
\end{thm}
\begin{proof}
Let $\calB_0:\frakF''_4\to\calC'(\frakH^\star)$ be the bijective map of Theorem \ref{thm-main}. There also exists a bijection $\calB_1:\frakF''_4\to\C_*\times(\C\setminus\R)$; this can be defined as follows.
If $\fp\in\frakF''_4$, we normalise so that 
$$
p_1=(1,\tan a),\quad p_2=\infty,\quad p_3=(0,0),\quad p_4=(z,t),
$$
where $a\neq\pm\pi/2$, $z\neq 0$ and also $\tan a\neq t/|z|^2$. We let $\calB_1$ be given by
$$
[\fp]\mapsto \left(z,\;\frac{|z|^2-it}{1-i\tan a}\right).
$$
This map is a bijection: Given $(\zeta,w)\in\C_{*}\times(\C\setminus\R)$,
we set
$$
z=\zeta,\quad \tan a=\frac{|\zeta|^2-\Re(w)}{\Im(w)},\quad t=\frac{\Re(w)|\zeta|^2-|w|^2}{\Im(w)}.
$$
To conclude the proof, let $F=\calB_1\circ \calB_0^{-1}:\calC'(\frakH^\star)\to\C_*\times(\C\setminus\R)$ be given by
$$
F(z,t,r)=\left(z,\;\frac{|z|^2-it}{1-i\log r}\right),\quad (z,t,r)\in\calC'(\frakH^\star).
$$
We will show that $F$ is CR with respect to $\calH^\star=\{\bX,\bY\}$, i.e. $F_{*}(\bZ)\in{\rm T}^{(1,0)}(\C_*\times(\C\setminus\R))$.

One can know that $F$ is bijective, because
%\end{prop}
%\begin{proof}
%First, $F$ is bijective; its inverse is given by
$$
F^{-1}(\zeta,w)=\left(\zeta,\;\frac{\Re(w)|\zeta|^2-|w|^2}{\Im(w)},\;e^{\frac{|\zeta|^2-\Re(w)}{\Im(w)}}\right),\quad(\zeta,w)\in\C_*\times(\C\setminus\R).
$$
Thus $F$ is a diffeomorphism and also 
\begin{eqnarray*}
F_*(\bZ)&=&\bZ(z)\frac{\partial}{\partial\zeta}+
\bZ(\overline{z})\frac{\partial}{\partial\overline{\zeta}}+
\bZ\left(\frac{|z|^2-it}{1-i\log r}\right)\frac{\partial}{\partial w}+
\bZ\left(\frac{|z|^2+it}{1+i\log r}\right)\frac{\partial}{\partial \overline{w}}\\
&=&\zeta\frac{\partial}{\partial\zeta}-2i\frac{|\zeta|^2\Im(w)}{\overline{w}-|\zeta|^2}\frac{\partial}{\partial w},
\end{eqnarray*}
which proves our claim about the CR equivalence. Now,
\begin{eqnarray*}
F_*(\bW)&=&\bW(z)\frac{\partial}{\partial\zeta}+
\bW(\overline{z})\frac{\partial}{\partial\overline{\zeta}}+
\bW\left(\frac{|z|^2-it}{1-i\log r}\right)\frac{\partial}{\partial w}+
\bW\left(\frac{|z|^2+it}{1+i\log r}\right)\frac{\partial}{\partial \overline{w}}\\
&=&-\Im\left(\zeta\frac{\partial}{\partial\zeta}+i\frac{w\Im(w)}{w-|\zeta|^2}\frac{\partial}{\partial w}\right)
\end{eqnarray*}
and thus the two complex structures are not biholomorphic.
\end{proof}

\textbf{Acknowledgement.}  Part of this work has been carried out while IDP was visiting Hunan University, Changsha, PRC. Hospitality is gratefully appreciated. We thank the refree for his/her valuable comments.

\end{document}